\documentclass[11pt]{amsart}
\usepackage{amsmath}
\usepackage{amsfonts}
\usepackage{amssymb}
\usepackage{amsthm}

\usepackage[margin=.72in]{geometry}
\usepackage[breaklinks=true]{hyperref}
\hypersetup{colorlinks=true, citecolor=red, linkcolor=blue}

\makeatletter
\def\@settitle{\begin{center}\Large\textbf{\@title}\end{center}}

\renewenvironment{abstract}{
\begin{center}%
{\vspace{.8\baselineskip}
\bfseries \abstractname
\vspace{.6\baselineskip}}
\end{center}\quotation
}

\renewcommand{\section}{\makeatletter
\renewcommand{\@seccntformat}[1]{{\csname the##1\endcsname.}\hspace{0.45em}}
\makeatother
\@startsection
{section}%                                            the name
{1}%                                                  the level
{0pt}%                                                the indent
{1.1\baselineskip}%                                      the beforeskip
{1.1\baselineskip}%                                   the afterskip
{\centering\normalsize\bfseries\mathversion{bold}}}%  the style
\renewcommand{\subsection}{\makeatletter
\renewcommand{\@seccntformat}[1]{{\csname the##1\endcsname.}\hspace{0.45em}}
\makeatother
\@startsection
{subsection}%                                         the name
{2}%                                                  the level
{0pt}%                                                the indent
{0.5\baselineskip}%                                   the beforeskip
{0.25\baselineskip}%                                  the afterskip
{\normalsize\bfseries\mathversion{bold}}}%            the style
\renewcommand{\subsubsection}{\makeatletter
\renewcommand{\@seccntformat}[1]{{\rm{\csname the##1\endcsname.}}\hspace{0.45em}}
\makeatother
\@startsection
{subsubsection}%                   the name
{3}%                               the level
{0pt}%                             the indent
{0.5\baselineskip}%                the beforeskip
{0.2\baselineskip}%              the afterskip
{\normalsize\bfseries\mathversion{bold}}}%            the style

\makeatother

\usepackage{bm}

\usepackage{enumitem}

\newtheorem{tm}{Theorem}[section]

\newtheorem{lm}[tm]{Lemma} 
\newtheorem{co}[tm]{Corollary}

\newtheorem{ex}[tm]{Example}
\newtheorem{re}[tm]{Remark}

\theoremstyle{definition} 
\newtheorem{df}[tm]{Definition}

\renewcommand{\P}[1]{\mathbb{P}\left(#1\right)}
\newcommand{\E}[1]{\mathbb E\left[#1\right]}

\newcommand{\X}{\mathbf X}
\newcommand{\Y}{\mathbf Y}
\newcommand{\Z}{\mathbf Z}
\newcommand{\R}{\mathbb R}

\renewcommand{\c}{\mathbf c}
\newcommand{\x}{\mathbf x}

\newcommand{\T}{\mathcal T}

\renewcommand{\ge}{\geqslant}
\renewcommand{\le}{\leqslant}

\newcommand{\Om}{\mathbf\Omega}
\newcommand{\om}{\bm\omega}
\newcommand{\ome}{\omega}
\newcommand{\I}[1]{\mathbf 1_{\left\{#1\right\}}}
\newcommand{\Exp}[1]{\exp\left(#1\right)}
\newcommand{\m}{\mid}

\title{When Janson meets McDiarmid: Bounded difference inequalities \\
under graph-dependence}

\author{}
\date{}

\begin{document}

\maketitle

\begin{center}
Rui-Ray Zhang \\
rui.zhang@monash.edu \\
School of Mathematics, Monash University
\end{center}

\begin{abstract}
\small{
We establish concentration inequalities for 
Lipschitz functions of dependent random variables,
whose dependencies are specified by forests.
We also give concentration results for decomposable functions,
improving Janson's Hoeffding-type inequality for the summation of graph-dependent bounded variables.
These results extend McDiarmid's bounded difference inequality 
to the dependent cases.}

% \textbf{Keywords}: concentration inequalities; dependency graphs. \\
% \textbf{AMS MSC 2010}: 60C05; 60F10.
\end{abstract}

\section{Introduction}

Concentration inequalities 
%are fundamental tools in many fields,
%such as probability theory, statistics, and probabilistic combinatorics. 
%They 
bound the deviation of a function of random variables 
from some value that is usually the expectation,
see \cite{boucheron2013concentration} for a good reference.
One of the well-known ones, bounded difference inequality 
(also called McDiarmid's inequality or Azuma-Hoeffding inequality)
gives exponential concentration bound
for Lipschitz functions of independent random variables.

McDiarmid's inequality requires independence, 
thus is restrictive in certain applications.
We extend it to the dependent cases via dependency graph, 
which is a common combinatorial tool for modelling the dependencies among random variables.
The dependency graph has been widely used in the probability and statistics
to establish normal approximation or Poisson approximation
via the Stein's method, cumulants, etc.
(see, for example, \cite{baldi1989normal, janson1990poisson, janson1988normal, isaev2020extreme}).
It is also heavily used in probabilistic combinatorics, 
such as Lov\'asz local lemma \cite{erdos1975problems}, 
Janson's inequality \cite{janson1988exponential}, etc.

In this note, we use the standard graph-theoretic notations. 
All graphs considered are finite, undirected and simple.
Given a graph $G = (V, E)$, 
let $V(G)$ be the vertex set and $E(G)$ be the edge set.
The edge connecting a pair of vertices $u, v$ is denoted as $\{ u, v \}$,
which is assumed to be unordered.
For every $S \subseteq V(G)$,
the induced subgraph of $G$ by $S$ is denoted as $G[S]$, 
that is, for any two vertices $u, v \in S$, $u, v$ are adjacent in $G[S]$ 
if and only if they are adjacent in $G$.
A tree is a connected, acyclic graph, and a forest is a disjoint union of trees.

Throughout this note, let $n$ be a positive integer 
and $[n]$ be the set $\{1,2, \ldots, n\}$. 
Let $\Omega_i$ be a Polish space for all $i\in [n]$,  
$\Om = \prod_{i\in[n]} \Omega_i = \Omega_1 \times \ldots \times \Omega_n $ be the product space, 
$\R$ be the set of real numbers, 
and $\R_+$ be the set of non-negative real numbers.
Let $\|\cdot\|_p$ denote the standard $\ell_p$-norm of a vector.
We use uppercase letters for random variables, lowercase letters for their realizations,
and bold letters for vectors.
For every set $V \subseteq [n]$,  
let $\Om_V = \prod_{i \in V} \Omega_i$,
$\X_V = ( X_i )_{i\in V}$, 
and $\x_V = ( x_i )_{i \in V}$.

We first introduce the definition of a Lipschitz function.
\begin{df}[$\c$-Lipschitz]
Given a vector $\c= (c_1, \ldots, c_n) \in \R_+^n$, 
a function $f:\Om \rightarrow \R$ is $\c$-Lipschitz 
if for all $\x = (x_1, \ldots, x_n)$ and $\x' = (x'_1, \ldots, x'_n)\in \Om$,
\begin{align}
| f(\x) - f(\x') |
\le \sum_{i = 1}^n c_i \I{ x_i \ne x'_i },
\label{lip}
\end{align}
where $c_i$ is the $i$-th Lipschitz coefficient of $f$
(with respect to the Hamming metric).
\end{df}

McDiarmid's inequality states that 
a Lipschitz function of independent random variables concentrates around its expectation.

\begin{tm}[McDiarmid's inequality~\cite{mcdiarmid1989method}]
Let $f: \Om \rightarrow \R$ be $\c$-Lipschitz 
and $\X = (X_1, \ldots, X_n)$ be a vector of independent random variables 
that takes values in $\Om$.
Then for every $t>0$,
\begin{align}
\P { f(\X) - \E { f(\X) } \ge t }
\le \exp \left( - \dfrac{2t^2}{ \| \c \| ^2_2} \right).
\label{mc-ie}
\end{align}
\label{McDiarmid's inequality}
\end{tm}

We extend McDiarmid's inequality to the graph-dependent case,
where the dependencies among random variables are characterized by a dependency graph.

%\begin{df}[Dependency graph]
%A graph $G$ is a dependency graph 
%of a random vector $\X= (X_1, \ldots, X_n)$ if:
%\begin{itemize}
%\item[(1)] $V(G) = [n]$.
%\item[(2)] For any disjoint $S, T \subset [n]$, if $S$ and $T$ are non-adjacent in $G$,
%then $\{ X_i \}_{i \in S}$ and $\{X_j \}_{j \in T}$ are independent.
%\end{itemize}
%We say that $\X$ is $G$-\textit{dependent}.
%\label{dep graph}
%\end{df}

\begin{df}[Dependency graph]
Given a graph $G = (V, E)$,
we say that a random vector $\X = (X_i)_{i \in V}$ is $G$-\textit{dependent}
if for any disjoint $S, T \subset V$ such that $S$ and $T$ are non-adjacent in $G$
(that is, no edge in $E$ has one endpoint in $S$ and the other in $T$),
random variables $\{ X_i \}_{i \in S}$ and $\{X_j \}_{j \in T}$ are independent.

\label{dep graph}
\end{df}

The above dependency graph is a strong version; there are ones 
with weaker assumptions, such as the one used in Lov\'asz local lemma.
Let $K_n$ denote the complete graph on $[n]$,
that is, every two vertices are adjacent.
%$\{i, j\} \in \binom{[n]}{2}$ ,
%where $\binom{V}{r}$ 
%denotes the family of $r$-element subsets of set $V$
%for every $r \in [|V|]$.
Then $K_n$ is a dependency graph 
for any set of variables $( X_i )_{i \in [n]}$.
Note that the dependency graph for a set of random variables 
may not be necessarily unique
and the sparser ones are the more interesting ones.
%In this note, we assume the graph is given.
The term `$G$-dependent' (graph-dependent) first appeared in \cite{petrovskaya1983central},
and various other notions such as `locally dependent' \cite{chen2004normal}, 
`partly dependent' \cite{janson2004large}, etc. 
are essentially referring to the graph-dependence.

Janson obtained a Hoeffding-type inequality for graph-dependent random variables
by breaking up the sum into sums of independent variables.

\begin{tm}[Janson's concentration inequality~\cite{janson2004large}]
Let random vector $\X$ be $G$-dependent such that for every $i \in V(G)$, 
random variable $X_i$ takes values in a real interval of length $c_i \ge 0$.
Then, for every $t > 0$, 
\begin{align}
\P{ \sum_{i \in V(G)} X_i - \E { \sum_{i \in V(G)} X_i } \ge t }
\le \Exp{ - \dfrac{2t^2}{ \chi_f(G) \| \c \| ^2_2} },
\label{eq:janson}
\end{align}
where $\c= (c_i )_{i \in V(G)}$ and $\chi_f(G)$ is 
the fractional chromatic number of $G$.
\label{Janson inequality}
\end{tm}

A fractional coloring of a graph $G$ is a mapping $g$ from $\mathcal I(G)$, 
the set of all independent sets of $G$, to $[0, 1]$ such that 
$\sum_{I \in \mathcal I(G): v \in I} g(I) \ge 1$ for every vertex $v \in V(G)$. 
The fractional chromatic number $\chi_f(G)$ of $G$ is the minimum of the value 
$\sum_{I \in \mathcal I(G)} g(I)$ over fractional colorings of $G$.

%Section \ref{discuss}

%A simple extension of the Janson's result to functions that is 
%the summation of some functions of independent random variables
%appeared in \cite{usunier2006generalization} 
%via an almost identical proof \cite{usunier2006apprentissage}.

\section{Results}

Here we introduce our concentration results for 
Lipschitz functions of dependent random variables,
whose dependencies are specified by forests
and for decomposable Lipschitz functions of general graph-dependent variables.

\subsection{Concentration under forest-dependence}

Our first result is for the case where the dependency graph is a tree.
\begin{tm}
Let function $f: \Om\rightarrow \R$ be $\c$-Lipschitz 
and $\Om$-valued random vector $\X$ be $G$-dependent.
If $G$ is a tree,
then for every $t > 0$,
\begin{align*}
%\E{ \Exp{ t( f(\X) - \E { f( \X ) } )} }
%%M_{f(\X) - \E { f( \X ) }}(t)
%&\le \Exp{ \frac{ t^2 }{ 8 } \left( c_{\min}^2 + \sum_{ \{i, j\} \in E(G) } ( c_i + c_j)^2 \right) }, \\
\P{ f(\X) - \E { f(\X) } \ge t }
&\le \Exp{ - \dfrac{2t^2}{ c_{\min}^2 + \sum_{ \{i, j\} \in E(G) } ( c_i + c_j)^2 } },
%\label{eq:trees}
\end{align*}
%\begin{align*}
%\E{ \Exp{ s( f(\X) - \E { f( \X ) } )} }
%\le \Exp{ \frac{ s^2 }{ 8 } \left( c_{\min}^2 + \sum_{ \{i, j\} \in E(G) } ( c_i + c_j)^2 \right) }.
%\end{align*}
%Moreover, for every $t > 0$, 
%\begin{align}
%\P{ f(\X) - \E { f(\X) } \ge t }
%\le \Exp{ - \dfrac{2t^2}{ c_{\min}^2 + \sum_{ \{i, j\} \in E(G) } ( c_i + c_j)^2 } },
%\label{eq:tree}
%\end{align}
where $c_{\min}$ is the minimum entry of $\c$.
%and
%\begin{align*}
%M_{X}(t) = \E{ \Exp{ tX } }
%\end{align*}
%denotes the moment generating function of a random variable $X$.
\label{tree}
\end{tm}

%Applying Theorem \ref{tree} to $-f$, 
%the same bound holds for the lower tail probability. 
%Combining lower and upper tails gives the two-sided probability.
%This is valid for other concentration inequalities in this note as well.

A simple extension of the proof leads to our second result, 
in which the dependency graph is a forest.
\begin{tm}
Let function $f: \Om\rightarrow \R$ be $\c$-Lipschitz 
and $\Om$-valued random vector $\X$ be $G$-dependent.
If $G$ is a disjoint union of trees $\{ T_i \}_{i\in[k]}$, 
then for $t > 0$,
\begin{align}
%\E{ \Exp{ t( f(\X) - \E { f( \X ) } )} }
%&\le \Exp{ \frac{ t^2 }{ 8 } \left( 
%\sum^k_{i=1} c_{\min,i}^2 + \sum_{ \{i, j\} \in E(G) } ( c_i + c_j)^2 \right) }, 
%\label{fmgf} \\
\P{ f(\X) - \E { f(\X) } \ge t }
&\le \Exp{ - \dfrac{2t^2}{ \sum^k_{i=1} c_{\min,i}^2 + \sum_{ \{i, j\} \in E(G) } ( c_i + c_j)^2 } },
\label{forest-ie}
\end{align}
where $c_{\min,i} := \min \{ c_j : j \in V(T_i) \}$ for all $i \in [k]$.
%Moreover, for every $t > 0$, 
%\begin{align}
%\P{ f(\X) - \E { f(\X) } \ge t }
%\le \Exp{ - \dfrac{2t^2}{ \sum^k_{i=1} c_{\min,i}^2 + \sum_{ \{i, j\} \in E(G) } ( c_i + c_j)^2 } }.
%\label{forest-ie}
%\end{align}
\label{forest}
\end{tm}

\begin{re}
If random variables $(X_1, \ldots, X_n)$ are independent,
then the empty graph $\overline{K}_n = ([n], \emptyset)$ is a valid dependency graphs for $(X_i)_{i\in[n]}$. 
In this case, inequality (\ref{forest-ie})
becomes the McDiarmid's inequality
(Theorem~\ref{McDiarmid's inequality}),
since each vertex is treated as a tree.

If all Lipschitz coefficients are of the same value $c$, then 
the denominator of the exponent
in (\ref{forest-ie}) becomes $ k c^2 + 4 (n-k) c^2 = (4n - 3k) c^2 $,
since the number of edges in the forest is $n - k$.
The denominator in Janson's bound (\ref{eq:janson}) is $2nc^2$,
since the fractional chromatic number of any tree is $2$.
Thus if $k \ge 2n/3$, then (\ref{forest-ie}) is better than Janson's bound in this case.
\label{re}
\end{re}

\subsection{Concentration of decomposable functions via fractional vertex coverings}
\label{frac}

For Lipschitz functions of general graph-dependent random variables, 
we give concentration results under certain decomposability constraints.

First we introduce the forest vertex covering and independent vertex covering of a graph.
Formally, given a graph $G$, we introduce the following.

\begin{enumerate}[label=(\alph*)]
\item A family $\{ S_{k} \}_{k}$ of subsets of $V(G)$ is a vertex cover of $G$ if $\bigcup S_{k} = V(G)$.

\item A family $\{ ( S_{k}, w_{k} ) \}_{k}$ of pairs $( S_{k}, w_{k} )$, where $S_{k} \subseteq V(G)$
and $w_{k} \in [0, 1]$ is a fractional vertex cover of $G$ 
if $\sum_{k: v \in F_{k}} w_{k} = 1$ for every $v \in V(G)$.

\item A fractional forest vertex cover $\{ ( F_{k}, w_{k} ) \}_{k}$ of $G$ is 
a fractional vertex cover
such that each set $F_{k}$ in it induces a forest of $G$.
We denote the set of (vertex sets of) disjoint trees in forest $F_{k}$ as $\T(F_k)$.
The set of all fractional forest vertex cover of graph $G$ is denoted as $ \textsf{FFC}(G) $.

\item A fractional independent vertex cover $\{ ( I_{k}, w_{k} ) \}_{k}$ of $G$ is 
a fractional vertex cover
such that $I_{k} \in \mathcal I(G)$ for every $k$.
%each set $I_{k}$ in it is an independent set of $G$.
The set of all fractional independent vertex cover of graph $G$ is denoted as $ \textsf{FIC}(G) $.
\end{enumerate}

Note that the fractional chromatic number $\chi_f(G)$ of graph $G$
is the minimum of $\sum_{k} w_{k}$ over $\textsf{FIC}(G)$ 
(see, for example, \cite{janson2004large}).

Next, we introduce the decomposable Lipschitz functions.
Given a graph $G$ on $n$ vertices and a vector $\c= (c_i)_{i \in [n]} \in \R_+^n$,
a function $f:\Om \rightarrow \R$ is forest-decomposable $\c$-Lipschitz
with respect to graph $G$ if for all $\x = (x_1, \ldots, x_n) \in \Om$
and for all $\{ ( F_{k}, w_{k} ) \}_{k} \in \textsf{FFC}(G)$,
there exist $(c_{i})_{i \in F_{k}}$-Lipschitz functions $\{ f_{k}: \Om_{F_{k}} \rightarrow \R \}_{k}$ such that
$
f( \x ) = \sum_{ k } w_k f_{k}( \x_{F_k} ).
$

\begin{tm}
Let $\Om$-valued random vector $\X$ be $G$-dependent
and function $f: \Om\rightarrow \R$ be forest-decomposable $\c$-Lipschitz with respect to $G$.
Then for every $t > 0$, 
\begin{align}
\P{ f(\X) - \E { f(\X) } \ge t }
\le \Exp{ - \dfrac{2t^2}{ D(G, \c) } },
\end{align}
where 
\begin{align}
D(G, \c) := \min_{\{ ( F_{k}, w_{k} ) \}_{k} \in \textsf{FFC}(G)} 
\left( \sum_{k} w_k \sqrt{ \sum_{\{i, j\} \in E(G[F_k])} ( c_i + c_j )^2 
+ \sum_{T \in \T(F_k)} c_{\min,k,T}^2 } \right)^2,
\label{T}
\end{align}
and $c_{\min,k,T} := \min \{ c_i : i \in T\}$ for all $T \in \mathcal T(F_k)$.
\label{decom}
\end{tm}

\begin{re}
An upper bound for $D(G, \c)$ via fractional chromatic number
follows an approach by Janson \cite{janson2004large}.
Let $\{ ( I_{k}, w_{k} ) \}_k \in \textsf{FIC}(G)$ be 
a fractional independent vertex cover of $G$.
Since
$\textsf{FIC}(G) \subseteq \textsf{FFC}(G)$ for all graph $G$, 
then
\begin{align*}
D(G, \c) 
%&= \min_{\{ ( F_{k}, w_{k} ) \}_{k} \in \textsf{FFC}(G)} \left( \sum_{k} w_k \sqrt{
%\sum_{ \{i, j\} \in E(F_k)} ( c_i + c_j )^2 + \sum_{j} c_{\min,k,j}^2} \right)^2 \\
\le 
%\min_{\{ ( F_{k}, w_{k} ) \}_{k} \in \textsf{FIC}(G)} 
\left( \sum_{k} w_k \sqrt{
\sum_{\{i, j\} \in E(G[I_k])} ( c_i + c_j )^2 + \sum_{T \in \T(I_k)} c_{\min,k,T}^2} \right)^2 
=
%\min_{\{ ( I_{k}, w_{k} ) \}_{k} \in \textsf{FIC}(G)} 
\left( \sum_{k} w_k \sqrt{ \sum_{i \in I_k} c_{i}^2} \right)^2,
\end{align*}
where the equality is because 
$E(G[I_k]) = \emptyset$,
since $I_{k} \in \mathcal I(G)$ for every $k$.
%$\{ ( I_{k}, w_{k} ) \}_{k} \in \textsf{FIC}(G)$.
Next by the Cauchy-Schwarz inequality,
\begin{align*}
\left( \sum_{k} w_k \sqrt{ \sum_{i \in I_k} c_{i}^2} \right)^2
\le \left( \sum_{k} w_k \right)
\left( \sum_{k} w_k \sum_{i \in I_k} c_i^2 \right)
= \left( \sum_{k} w_k \right)
\left( \sum_{i \in V(G)} \sum_{k: i \in I_k} w_i c_i^2 \right) 
= \left( \sum_{k} w_k \right)
\sum_{i \in V(G)} c_i^2.
\end{align*}
Then by choosing $\{ ( F_{k}, w_{k} ) \}_{k} \in \textsf{FIC}(G)$
with $\sum_{k} w_{k} = \chi_f(G)$,
we have
$
D(G, \c)
\le \chi_f(G) \| \c \| ^2_2,
$
which is exactly the denominator in Janson's bound in (\ref{eq:janson}).
\label{jan}
\end{re}

\subsubsection{Concentration of the sum of graph-dependent random variables}

Here we give an application.
This improves Janson's
Hoeffding-type inequality for graph-dependent random variables.

\begin{co}
Let random vector $\X$ be $G$-dependent.
If for every $i \in V(G)$, 
random variable $X_i$ takes values in a real interval of length $c_i \ge 0$,
then for every $t > 0$, 
\begin{align}
\P{ \sum_{i \in V(G)} X_i - \E { \sum_{i \in V(G)} X_i } \ge t }
\le \Exp{ - \dfrac{2t^2}{ D(G, \c) } },
\end{align}
where $\c= ( c_v )_{v \in V(G)}$ and $ D(G, \c) $ is defined by $(\ref{T})$.
\label{sum}
\end{co}
\begin{proof}
It suffices to show that the summation is forest-decomposable $\c$-Lipschitz.
Since for every $\{ ( F_{k}, w_{k} ) \}_{k} \in \textsf{FFC}(G)$, we have
$
\sum_{i \in V(G)} X_i
= \sum_{i \in V(G)} \sum_{k: i \in F_k} w_i X_i
= \sum_{k} w_i \sum_{i \in F_k} X_i,
$
then Corollary \ref{sum} follows from Theorem \ref{decom}.
\end{proof}

Next we give an example in which our bound is better than Janson's.
\begin{ex}
%Here we give an example by considering U-statistic of 
%i.i.d. random variables $(X_i)_{i \in [4]}$ of degree $3$, that is,
%\begin{align*}
%U = \sum_{ \{ i_1, i_2, i_3 \} \in \binom{[4]}{3} } f_{i_1, i_2, i_3}( X_{i_1}, X_{i_2}, X_{i_3} ).
%\end{align*}
%This is a sum of $K_4$-dependent random variables.
%If the kernel function $f_{i_1, i_2, i_3}$ takes values in a real interval of length $c \ge 0$
%for every $\{ i_1, i_2, i_3 \} \in \binom{[4]}{3}$, then
%Janson's result \cite[Eq. (4.6)]{janson2004large} gives
%\begin{align}
%\P{ U - \E { U } \ge t }
%\le \Exp{ - \dfrac{t^2}{ 8 c^2 } },
%\end{align}
%we give a slightly better bound 
%\begin{align}
%\P{ U - \E { U } \ge t }
%\le \Exp{ - \dfrac{t^2}{ 5c^2 } }.
%\end{align}
%This is by giving equal weights to three perfect matchings of $K_4$,
%that is, three pairs of two non-adjacent edges.
%Then
%\begin{align*}
%\left( \sum_{k} w_k \sqrt{
%\sum_{\{i, j\} \in E(G[F_k])} ( c_i + c_j )^2 + \sum_{T \in \T(F_k)} c_{\min,k,T}^2} \right)^2
%= ( c \sqrt{ 2^2 + 2^2 + 1 + 1 } )^2 = 10 c^2.
%\end{align*}

Let $( X_1, X_2, X_3 )$ be dependent random variables with $K_3$ as their dependency graph,
and $( X_i )_{4 \le i \le 9} $ be independent variables that are also independent from $( X_i )_{i \in [3]}$.
Then the vertex-disjoint union of $K_3$ and $6$ copies of $K_1$ 
is a dependency graph for $( X_i )_{i \in [9]}$.
If for every $i \in [9]$, 
random variable $X_i$ takes values in a real interval of length $c \ge 0$,
then for every $t > 0$, Janson's bound gives
\begin{align*}
\P{ \sum_{i \in [9]} X_i - \E { \sum_{i \in [9]} X_i } \ge t }
\le \Exp{ - \dfrac{2 t^2}{ 27 c^2 } },
\end{align*}
since $\chi_f = 3$.
We give a slightly better bound 
\begin{align*}
\P{ \sum_{i \in [9]} X_i - \E { \sum_{i \in [9]} X_i } \ge t }
\le \Exp{ - \dfrac{ 8t^2}{ 81 c^2 } }.
\end{align*}
This is by giving equal weight $1/2$ to vertex covers $F_1 = \{ 1, 2, 4, 5, 6, 7\}$,
$F_2 = \{ 1, 3, 4, 5, 8, 9 \}$ and $F_3 = \{ 2, 3, 6, 7, 8, 9 \}$.
Then the subgraph induced by $F_k$ in $G$ is 
a vertex-disjoint union of $K_2$ and $4$ copies of $K_1$
% $G[F_k] = K_2 \oplus K_1 \oplus K_1 \oplus K_1 \oplus K_1$ 
for all $k \in [3]$, 
thus we have
$
\left( \sum_{k} w_k \sqrt{
\sum_{\{i, j\} \in E(G[F_k])} ( c_i + c_j )^2 + \sum_{T \in \T(F_k)} c_{\min,k,T}^2} \right)^2
= \left( \frac{3c}{2} \sqrt{ 2^2 + 1 + 4 } \right)^2 =  81 c^2/4.
$

\end{ex}

%--------------------------------
%
%\subsection{Upper bound via fractional vertex arboricity}
%
%Let $\{ ( F_{k}, w_{k} ) \}_{k \in [K]}$ be a fractional forest cover of $G$.
%By the Cauchy-Schwarz inequality,
%\begin{align}
%D(G, \c)
%&\le 
%\left( \sum_{k} w_k \sqrt{\sum_{i \in F_k} ( c_i + c_{p_i} )^2 + \sum_{j} c_{\min,k,j}^2} \right)^2 \\
%&\le \left( \sum_{k \in [K]} w_k \right)
%\left( \sum_{k \in [K]} w_k 
%\left(\sum_{i \in F_k} ( c_i + c_{p_i} )^2 + \sum_{j} c_{\min,k,j}^2 \right) \right).
%%&= a_f(G)
%%\left( \sum_{k \in [K]} w_k \sum_{i \in F_k} ( c_i + c_{p_i} )^2 
%%+ \sum_{k \in [K]} w_k c_{\min,k}^2 \right)
%\end{align}
%Since
%\begin{align}
%\sum_{k \in [K]} w_k \sum_{j} c_{\min,k,j}^2
%\le \sum_{k \in [K]} w_k \sum_{i \in F_k} c_k^2
%= \sum_{i \in V(G)} \sum_{k \in [K]: i \in F_k} w_k c_i^2
%= \sum_{i \in V(G)} c_i^2,
%\end{align}
%and
%\begin{align}
%\sum_{k \in [K]} w_k \sum_{i \in F_k} ( c_i + c_{p_i} )^2
%\le 2 \sum_{k \in [K]} w_k \sum_{i \in F_k} ( c_i^2 + c_{p_i}^2 )
%= 2 \sum_{k \in [K]} w_k \sum_{i \in F_k} d_{F_k}(i) c_i^2
%\le 2 \Delta \sum_{i \in V(G)} c_i^2,
%\end{align}
%then we have
%\begin{align}
%D(G, \c) 
%\le \left( \sum_{k \in [K]} w_k \right)
%(2\Delta+1) \sum_{i \in V(G)} c_i^2
%= a_f(G) (2\Delta+1) \| \c \| ^2_2.
%\end{align}
%since
%$ a_f(G) \le \lceil \frac{\Delta+1}{2} \rceil$, then
%\begin{align*}
%D(G, \c) \le \lceil \frac{\Delta+1}{2} \rceil (2\Delta+1) \| \c \| ^2_2
%\end{align*}
%
%--------------------------------

\subsection{Concentration under local dependence}
\label{sec-m}

A sequence of random variables $( X_i )_{i = 1}^n$ is said to be $f(n)$-dependent 
if subsets of variables separated by some distance $f(n)$ are independent.
This was introduced by Hoeffding and Robbins \cite{hoeffding1948central}
and has been studied extensively (see, for example, \cite{stein1972bound, chen1975poisson}).
This is usually the canonical application for the results based on the dependency graph model.
A special case of $f(n)$-dependence when $f(n) = m$ is the following $m$-dependent model.

\begin{df}[$m$-dependence~\cite{hoeffding1948central}]
A sequence of random variables $( X_i )_{i = 1}^n$ is $m$-dependent for some $m \ge 1$
if $( X_j )_{j = 1}^i$ and $( X_j )_{j = i+m+1}^n$
are independent for all $i  > 0$.
\label{m-dependence result}
\end{df}

The $m$-dependent sequences usually appear as block factors. 
Let $k \in \mathbb N$, the sequence $( X_i )$ is an $k$-block factor 
if there are an independent and identically distributed sequence $( Y_j )_{-\infty}^\infty$ 
and a function $g: \R^k \rightarrow \R$ such that 
$X_i = g(Y_i, \ldots, Y_{i + k - 1})$.
%i.e. the dependences are from sharing underlying independent random variables.
Note that every such sequence $( X_i )$ is $(k - 1)$-dependent,
and there are $m$-dependent sequences that are not block factors, see \cite{burton19931}.

%Block factors are widely used in statistics and combinatorics, 
%e.g. moving-average model and random binary search tree, etc.

\begin{co}
Let $f: \Om\rightarrow \R$ be $\c$-Lipschitz 
and $\Om$-valued random vector $\X$ be $m$-dependent.
Then for every $t > 0$, 
\begin{align}
\P{ f(\X) - \E { f(\X) } \ge t }
\le \Exp{ - \dfrac{ 2t^2 }
{ \sum_{i \in \left[ \left\lfloor\frac{n}{m}\right\rfloor \right] }
\left( \sum_{ j \in B_i \cup B_{i+1} }  c_j \right)^2
+ \min_{i \in \left[ \left\lceil\frac{n}{m}\right\rceil \right] }
\left( \sum_{ j \in B_i }  c_j \right)^2  } },
\label{m-ie}
\end{align}
where for every $j \in \left[ \left\lfloor\frac{n}{m}\right\rfloor \right]$,
\begin{align}
B_j := \{ k:  (j-1)m +1 \le k \le jm \},
\quad\text{ and }\quad
B_{\left\lceil\frac{n}{m}\right\rceil} 
:= [n] \setminus \cup_{j \in \left[ \left\lfloor\frac{n}{m}\right\rfloor \right]} B_j.
\label{B}
\end{align}
\label{m}
\end{co}

\begin{re}
Corollary \ref{m} improves the following bound obtained by Paulin 
in \cite[Example 2.14]{paulin2015concentration}: \\
\begin{align*}
\P{ f(\X) - \E { f(\X) } \ge t }
\le \Exp{ - \dfrac{ 2t^2 }
{ \sum_{i \in \left[ \left\lfloor\frac{n}{m}\right\rfloor \right] }
\left( \sum_{ j \in B_i \cup B_{i+1} }  c_j \right)^2
+ \left( \sum_{ j \in B_{  \left\lceil\frac{n}{m}\right\rceil } }  c_j \right)^2 } },
\end{align*}
% \begin{align*}
% \P{ f(\X) - \E { f(\X) } \ge t }
% \le \Exp{ - 2t^2
% \left( \sum_{i \in \left[ \left\lfloor\frac{n}{m}\right\rfloor \right] }
% \left( \sum_{ j \in B_i \cup B_{i+1} }  c_j \right)^2
% + \left( \sum_{ j \in B_{  \left\lceil\frac{n}{m}\right\rceil } }  c_j \right)^2 \right)^{-1} },
% \end{align*}
where the second summand in the denominator of the exponent is without taking minimum over blocks.
Note that Corollary \ref{m} does not assume stationarity,
and it may be slightly improved by choosing better grouping schemes.

If all Lipschitz coefficients are of the same value $c$ and w.l.o.g. assume that $n$ is divisible by $m$, 
then the denominator of the exponent
in \eqref{m-ie} becomes $(2mc)^2 (n/m - 1)+m^2c^2 \le 4mnc^2$,
thus we have 
\begin{align}
\P{ f(\X) - \E { f(\X) } \ge t }
\le \Exp{ - \dfrac{ t^2 } { 2mnc^2 } },
\end{align}
which is $4m$ times worse than the independent case, see \eqref{mc-ie}.
\end{re}

%\begin{align*}
%\sum_{i \in \left[ \left\lfloor\frac{n}{m}\right\rfloor \right] }
%\left( \sum_{ (i-1)m + 1 \le j \le \min( (i+1)m , n) }  c_j \right)^2
%+ \min_{i \in \left[ \left\lceil\frac{n}{m}\right\rceil \right] }
%\left( \sum_{ (i-1)m +1 \le j \le \min( im, n) }  c_j \right)^2
%= \end{align*}
%and it immediately implies the concentration of maxima under $m$-dependence.
%Let $( X_i )_{i = 1}^n$ be $m$-dependent random variables 
%and denote their maximum as $M_{n, m} := \max( X_1, \ldots, X_n )$, then for every $t > 0$,
%\begin{align*}
%\P{ M_{n, m} - \E { M_{n, m} } \ge t }
%\le \Exp{ - \dfrac{ 2t^2 }
%{ \left\lfloor\frac{n}{m}\right\rfloor( m+m )^2 + m^2 } }
%%\le \Exp{ - \dfrac{ 2t^2 }{ 4mn + m^2 } }
%\le \Exp{ - \dfrac{ 2t^2 }{ (4n+m)m } }.
%\end{align*}
%This exponential tail bound can be useful in extreme value theory.

\begin{proof}[Proof of Corollary~\ref{m}]
A dependency graph for 
$m$-dependent random variables $( X_i )_{i = 1}^n$
is
$ D_{n,m} = 
( [n], \{ \{ i, j \} : i, j \in [n], | i - j | \in [m] \} )$.
Note that $D_{n,m}$ is not a forest,
nevertheless, via the following transformation, we can apply our results.
We group the $m$-dependent random variables $( X_i )_{i = 1}^n$ into 
$\left\lceil\frac{n}{m}\right\rceil$ blocks 
such that each block $( X_i )_{i \in B_j}$ contains $m$ consecutive random variables 
except for the last one, which might contain less than $m$ ones,
where $( B_j )$ are defined in (\ref{B}).
The resulting dependency graph for the blocks 
$ ( ( X_i )_{i \in B_j} : j \in \left\lceil\frac{n}{m}\right\rceil )$ 
is a path $P$ on $\left\lceil\frac{n}{m}\right\rceil$ vertices. 
Since the Lipschitz coefficient $\widetilde c_j$ of each block $( X_i )_{i \in B_j}$ is at most 
$ \sum_{i \in B_j} c_i $ 
due to the triangle inequality, then we have
$\widetilde c_{\min}^2 + \sum_{ \{i, j\} \in E(P) } ( \widetilde c_i + \widetilde c_j )^2
= \sum_{i \in \left[ \left\lfloor\frac{n}{m}\right\rfloor \right] }
\left( \sum_{ j \in B_i \cup B_{i+1} }  c_j \right)^2
+ \min_{i \in \left[ \left\lceil\frac{n}{m}\right\rceil \right] }
\left( \sum_{ j \in B_i }  c_j \right)^2$.
The result then follows from the Theorem~\ref{tree}.
\end{proof}

\section{Proofs}

We first introduce some additional notations.
% Let $N_G(u) = \{v \in V(G): \{ u, v \} \in E(G)\}$ 
% denote the neighbors of a vertex $u$ in $G$, 
% and $N_G^+(u) = N_G(u)\cup\{u\}$ 
% For any set of vertices $V \subseteq V(G)$, 
% $N_G(V) := \cup_{v \in V} N_G (v) \setminus V $ and
% $N_G^+(V) := \cup_{v \in V} N^+_G (v) $
% denote the neighborhood and inclusive neighborhood of vertex set $V$.
Given a graph $G$, for every vertex $v \in V(G)$, 
let $N_G(v) := \{v \in V(G): \{ u, v \} \in E(G)\}$ 
denote the neighbours of $v$, and $N^{+}_G(v) := N_G(v) \cup \{ v \}$ 
denote the inclusive neighborhood.
The neighborhood of a set of vertices $V$ is $N_G^{+}(V) := \cup_{v \in V} N_G^{+} (v)$,
and the neighbours of $V$ is $N_G(V) := N_G^{+}(V) \setminus V $.
The subscript $G$ might be omitted if it is clear from context. 
Given $\Om = \prod_{i\in[n]} \Omega_i$,
let $\x = (x_1, \ldots, x_n)$ 
be an arbitrary vector in $\Om$
and $\x^{(i)} := (x_1, \ldots, x_{i-1}, x'_i, x_{i+1}, \ldots, x_n)$,
where $x'_i \in \Omega_i$.

For tree-dependent random variables $\X$, without loss of generality,
we assume that the dependency tree $G$ satisfies the following assumptions:

\begin{description}
\item[\textsf{Rooted}]: $G$ is rooted at the vertex $n$ and $c_n=c_{\min}$.
\item[\textsf{Ordered}]: for every pair of vertices $i, j \in V(G)$, 
$j$ is a descendant of $i$ only if $j < i$.
\end{description}

Notice the above assumptions are just for the simplicity 
of the statement of the vertex exposure martingale in the proofs,
and there is no such requirement for the ordering of the vertices.
Such ordering of tree vertices exists and can be obtained via a topological sort.

We briefly explain the idea before the formal proof.
The proof of Theorem \ref{tree} relies on Lemma \ref{McDiarmid's Lemma}, 
which states that the small deviation of 
%\begin{align*}
%\E { f(\X) \m \X_{[i-1]} = \x_{[i-1]}, X_i = x_i }
%\end{align*}
$
\E { f(\X) \m \X_{[i-1]} = \x_{[i-1]}, X_i = x_i }
$
with respect to $x_i \in \Omega_i$ for all $i \in [n]$
leads to the concentration of $f(\X)$ around its expectation. 
Our task is thus to bound the difference of the conditional expectations
$ \E { f(\X) \m \X_{[i-1]} = \x_{[i-1]}, X_i = \alpha }
- \E { f(\X) \m \X_{[i-1]} = \x_{[i-1]}, X_i = \beta }$
for any $\alpha, \beta \in \Omega_i$  and given $\X_{[i-1]}$ (Lemma \ref{martingale differences lemma}).
This is by the coupling constructions, namely,
jointly distributed variables $(\Y^{(i)}, \Z^{(i)})$ whose marginal distributions are 
distributions of $\X$ conditioned on 
$ \{ \X_{[i-1]} = \x_{[i-1]}, X_i = x_i \} $ 
and on 
$ \{ \X_{[i-1]} = \x_{[i-1]}, X_i = \x^{(i)}_i \}$, 
respectively.
Hence, the main part of the proof is to construct such
recursive couplings of the conditional probability distribution (Lemma \ref{coupling}) 
whose feasibility relies on the independence among $\X$ (Lemma \ref{lm:indepofxi}).

First of all, recall a lemma by McDiarmid. By this lemma, it suffices to bound the deviation of $\E { f(\X) \m \X_{[i-1]} = \x_{[i-1]}, X_i = x_i } $ 
with respect to $x_i \in \Om_i$ for all $i \in [n]$. 

\begin{lm}[\cite{mcdiarmid1989method}]
If for every $i \in [n]$, $\om_{[i-1]} \in \Om_{[i-1]}$,
there is a constant $c_i\ge 0$ such that 
\begin{equation}
\sup_{\alpha \in \Omega_i} \E { f(\X) \m \X_{[i-1]} = \om_{[i-1]} , X_i=\alpha }
- \inf_{\beta \in \Omega_i} \E { f(\X) \m \X_{[i-1]} = \om_{[i-1]} , X_i=\beta } \le c_i,
\end{equation}
then for $s > 0$,
\begin{align}
\E{ \Exp{ s( f(\X) - \E { f( \X ) } )} }
\le \Exp{ \frac{ s^2 }{ 8 } \sum c_i^2 }.
\label{mgf}
\end{align}
Moreover, the bound on moment-generating function (\ref{mgf}) implies that for $t > 0$,
\begin{align}
\P{ f(\X) - \E { f( \X ) } \ge t }
\le \Exp{ - \dfrac{2t^2}{ \sum_{i=1}^n c_i^2 } }.
\end{align}
\label{McDiarmid's Lemma}
\end{lm}

For every non-root vertex $i\in V(G)$, let $p_i$ be the parent vertex of $i$. 
For the rest of the section, define $S_i := [i+1,n] \setminus \{ p_i \}$, 
where $[j, k]$ stands for the integer set $\{j, \ldots, k\}$ for all $j < k$.
The following lemma indicates that the distribution of $\X_{S_i}$ 
is independent of the realization of $X_i$ when $\X_{[i-1]}$ is given.

\begin{lm}
For every $i \in [n - 1]$, $\om_{S_i} \in \Om_{S_i}$, we have
\begin{align*}
\P{ \X_{S_i} = \om_{S_i} \m \X_{[i]} = \x_{[i]} } 
= \P{ \X_{S_i} = \om_{S_i} \m \X_{[i]} = \x^{(i)}_{[i]}}.
\end{align*}
\label{lm:indepofxi}
\end{lm}

\begin{proof}
Let $T_i$ be the subtree rooted at vertex $i$ of the dependency tree $G$
(such $T_i$ is also called fringe subtree in the literature).
Since $G$ is assumed to be \textsf{Ordered}, we have $V(T_i) \subseteq [i]$,
and $[i] = V(T_i) \cup ( [i-1] \setminus V(T_i) )$.
We will actually show stronger results:
$\X_{V(T_i)}$ is independent of $\{\X_{S_i},\X_{[i-1]\setminus V(T_i)}\}$, 
which follows from the following two observations:

\noindent
\textbf{Observation 1}: $N^+_G(T_i) \cap ( [i-1] \setminus V(T_i) )= \emptyset$.
% where $N^+_G(T_i) = \cup_{k\in V(T_i)} N^+_G(k)$. 
Since $G$ is a dependency tree, then $N^+_G(T_i) = V(T_i) \cup \{ p_i \}$,
and $p_i \in [i+1, n]$
because $G$ is \textsf{Ordered}, then $p_i \not\in [i-1] \setminus V(T_i)$.
Thus $N^+_G(T_i) \cap ( [i-1] \setminus V(T_i) ) = \emptyset$.

\noindent
\textbf{Observation 2}: $N^+_G(T_i) \cap S_i = \emptyset$. 
From observation 1, $N^+_G(T_i) = V(T_i) \cup \{ p_i \}$.
Then observation 2 follows from the definition $S_i = [i+1,n] \setminus \{ p_i \}$.

Observations $1$ and $2$ indicate that 
$\X_{V(T_i)}$ is independent of $\{\X_{S_i},\X_{[i-1]\setminus V(T_i)}\}$,
due to the definition of the dependency graphs in Definition \ref{dep graph}.
Then
\begin{align*}
\P{ \X_{S_i} = \om_{S_i} \m \X_{[i-1]\setminus V(T_i)} = \x_{[i-1]\setminus V(T_i)} } 
&= \P{ \X_{S_i} = \om_{S_i} \m 
\X_{[i-1]\setminus V(T_i)} = \x_{[i-1]\setminus V(T_i)}, \X_{V(T_i)} = \x_{V(T_i)} } \\
&= \P{ \X_{S_i} = \om_{S_i} \m \X_{[i]} = \x_{[i]}} .
\end{align*}
Similarly, we also have 
\begin{align*}
& \P{ \X_{S_i} = \om_{S_i} \m \X_{[i-1]\setminus V(T_i)} = \x_{[i-1]\setminus V(T_i)} } \\
&= \P{ \X_{S_i} = \om_{S_i} \m 
\X_{[i-1]\setminus V(T_i)} = \x_{[i-1]\setminus V(T_i)}, 
\X_{V(T_i) \setminus \{i\}} = \x_{V(T_i) \setminus \{i\}}, X_i = x_i' }
= \P { \X_{S_i} = \om_{S_i} \m \X_{[i]} = \x^{(i)}_{[i]} }.
\end{align*}
The lemma follows from the combinations of the above.
\end{proof}

Then we introduce a Marton-type coupling \cite{marton2003measure, paulin2015concentration},
more precisely, we construct the joint distribution of random vector $(\Y^{(i)}, \Z^{(i)})$ 
taking values in $\Om \times \Om$, 
with respect to all $i \in [n - 1]$, $\x = (x_1, \ldots, x_n) \in \Om$,
and $\x^{(i)} = (x_1, \ldots, x_{i-1}, x'_i, x_{i+1}, \ldots, x_n)$,
where $x'_i \in \Omega_i$.
Specifically, the distributions of 
$\Y^{(i)} := (Y_1^{(i)}, \ldots, Y_n^{(i)})$ 
and $\Z^{(i)} := (Z_1^{(i)}, \ldots, Z_n^{(i)})$ are set as follows.

\begin{enumerate}[label=(C\arabic*)]
\item $\Y_{[i]}^{(i)} = \x_{[i]}$,
\item For every $ \om_{[i+1,n]} \in \Om_{[i+1,n]}$,
\begin{align}
\P {\Y_{[i+1,n]}^{(i)} = \om_{[i+1,n]} } 
= \P {\X_{[i+1,n]} = \om_{[i+1,n]} \m \X_{[i]} = \x _{[i]}}.
\label{Y}
\end{align}
\item $\Z_{[i]}^{(i)} = \x^{(i)}_{[i]}, \Z_{S_i}^{(i)} = \Y_{S_i}^{(i)}$.
\item For every $\om_{S_i} \in \Om_{S_i}$ and $\ome_{p_i} \in \Omega_{p_i}$,
\begin{align}
\P { Z_{p_i}^{(i)} = \ome_{p_i} \m \Z_{S_i}^{(i)} = \om_{S_i} }
=\P { X_{p_i} = \ome_{p_i} \m \X_{[i]} = \x^{(i)}_{[i]}, \X_{S_i} = \om_{S_i} }.
\label{Z}
\end{align}
\end{enumerate}

The next lemma states that $(\Y^{(i)}, \Z^{(i)})$ has the desired marginal distributions.
\begin{lm}\label{coupling}
For every $i \in [n - 1]$, $\om_{[i+1,n]} \in \Om_{[i+1,n]}$, we have 
\begin{enumerate}[label=(A\arabic*)]
\item $\P {\Y_{[i+1,n]}^{(i)} = \om_{[i+1,n]}}
= \P {\X_{[i+1,n]} = \om_{[i+1,n]} \m \X_{[i]} = \x_{[i]}} $,
\item $\P {\Z_{[i+1,n]}^{(i)} = \om_{[i+1,n]}} 
= \P {\X_{[i+1,n]} = \om_{[i+1,n]} \m \X_{[i]} = \x^{(i)}_{[i]}} $.
\end{enumerate}
\end{lm}
%\begin{proof}
%(A1) is by the constructions of $\Y^{(i)}$.
%For (A2), we arbitrarily choose 
%$\om_{[i+1,n]} = (\ome_{i+1}, \ldots, \ome_n) \in \Om_{[i+1,n]}$, then
%\begin{align*}
%\P {\Z_{[i+1,n]}^{(i)} = \om_{[i+1,n]} }
%&=\P {\Z_{S_i}^{(i)} = \om_{S_i} } \P {Z_{p_i}^{(i)} = \ome_{p_i} \m \Z_{S_i}^{(i)} = \om_{S_i}} \\
%&=\P {\Y_{S_i}^{(i)} = \om_{S_i} } \P {Z_{p_i}^{(i)} = \ome_{p_i} \m \Z_{S_i}^{(i)} = \om_{S_i}} \\
%&= \P {\X_{S_i} = \om_{S_i} \m \X_{[i]} = \x_{[i]}}
%\P { X_{p_i} = \ome_{p_i} \m \X_{[i]} = \x^{(i)}_{[i]},\X_{S_i} = \om_{S_i}} \\
%&= \P {\X_{S_i} = \om_{S_i} \m \X_{[i]} = \x^{(i)}_{[i]}}
%\P { X_{p_i} = \ome_{p_i} \m \X_{[i]} = \x^{(i)}_{[i]},\X_{S_i} = \om_{S_i} } \\
%&= \P { \X_{[i+1,n]} = \om_{[i+1,n]} \m \X_{[i]} = \x^{(i)}_{[i]} },
%\end{align*}
%where 
%the third equality is due to (\ref{Y}) and (\ref{Z})
%and the fourth equality is by Lemma~\ref{lm:indepofxi}.
%\end{proof}

\begin{proof}
(A1) is by the constructions of $\Y^{(i)}$.
For (A2), we arbitrarily choose 
$\om_{[i+1,n]} = (\ome_{i+1}, \ldots, \ome_n) \in \Om_{[i+1,n]}$, then
\begin{align*}
\P {\Z_{[i+1,n]}^{(i)} = \om_{[i+1,n]} }
&=\P {\Z_{S_i}^{(i)} = \om_{S_i} } \P {Z_{p_i}^{(i)} = \ome_{p_i} \m \Z_{S_i}^{(i)} = \om_{S_i}} \\
&=\P {\Y_{S_i}^{(i)} = \om_{S_i} } \P {Z_{p_i}^{(i)} = \ome_{p_i} \m \Z_{S_i}^{(i)} = \om_{S_i}}.
\end{align*}
Combining with \eqref{Y} and \eqref{Z} gives
\begin{align*}
\P {\Z_{[i+1,n]}^{(i)} = \om_{[i+1,n]} }
= \P {\X_{S_i} = \om_{S_i} \m \X_{[i]} = \x_{[i]}}
\P { X_{p_i} = \ome_{p_i} \m \X_{[i]} = \x^{(i)}_{[i]},\X_{S_i} = \om_{S_i}}.
\end{align*}
By Lemma~\ref{lm:indepofxi}, we have
\begin{align*}
\P {\Z_{[i+1,n]}^{(i)} = \om_{[i+1,n]} } 
&= \P {\X_{S_i} = \om_{S_i} \m \X_{[i]} = \x^{(i)}_{[i]}}
\P { X_{p_i} = \ome_{p_i} \m \X_{[i]} = \x^{(i)}_{[i]},\X_{S_i} = \om_{S_i} } \\
&= \P { \X_{[i+1,n]} = \om_{[i+1,n]} \m \X_{[i]} = \x^{(i)}_{[i]} }.
\end{align*}
This completes the proof.
\end{proof}

\begin{lm} 
For every $i \in [n - 1]$, we have
\begin{align*}
\E{ f(\X) \m \X_{[i]} = \x_{[i]} } - \E{ f(\X) \m \X_{[i]} = \x^{(i)}_{[i]} }
\le c_i + c_{p_i}.
\end{align*}
\label{martingale differences lemma}
\end{lm}
\begin{proof}
By the construction of random vectors $\Y^{(i)}, \Z^{(i)}$ and Lemma \ref{coupling}, we have
\begin{align*}
\E{ f(\X) \m \X_{[i]} = \x_{[i]}} - \E{ f(\X) \m \X_{[i]} = \x^{(i)}_{[i]} }
&= \E{ f(\Y^{(i)})} - \E{ f(\Z^{(i)}) }.
\end{align*}
By the linearity of expectation and the Lipschitz assumption (\ref{lip}), we get
\begin{align*}
\E{ f(\Y^{(i)})} - \E{ f(\Z^{(i)}) }
= \E{ f(\Y^{(i)}) - f(\Z^{(i)}) }
\le \E{ \sum^n_{j=1}c_j \I{Y_j \ne Z_j} }
\le c_i+c_{p_i},
\end{align*}
where the last inequality is because the only different variables of 
$\Y^{(i)}, \Z^{(i)}$ are the $i$-th and $p_i$-th ones due to the coupling construction.
\end{proof}

We are now ready to prove Theorem~\ref{tree} and Theorem~\ref{forest}.

\begin{proof}[Proof of Theorem~\ref{tree}]
Combining Lemma~\ref{McDiarmid's Lemma} and Lemma~\ref{martingale differences lemma}, we have
\begin{align*}
\P{ f(\X) - \E { f( \X ) } \ge t }
&\le \Exp{ -\dfrac{2t^2}{ c_n^2 + \sum_{i \in V(G) \setminus \{n\} }( c_i + c_{p_i} )^2 } } 
= \Exp{ -\dfrac{2t^2}
{ c_{\min}^2 + \sum_{ \{ i, j \} \in E(G)} ( c_i + c_j )^2 } },
\end{align*}
where the equality is due to the \textsf{Rooted} and \textsf{Ordered} assumptions.
\end{proof}

\begin{proof}[Proof of Theorem~\ref{forest}]
The proof is similar to that of Theorem~\ref{tree}. 
Without loss of generality, we assume that each tree $T_i$ of the forest $G$ 
is \textsf{Rooted} and \textsf{Ordered}. 
Then the proofs of Lemmas \ref{lm:indepofxi} - \ref{martingale differences lemma} remain valid, 
since variables in different connected components are independent. 
Then, the theorem follows from Lemma \ref{McDiarmid's Lemma}.
\end{proof}

\begin{proof}[Proof of Theorem~\ref{decom}]
Let $\{ ( F_{k}, w_{k} ) \}_{k \in [K]}$ be a fractional forest cover of $G$.
Since function $f$ is forest-decomposable $\c$-Lipschitz with respect to $G$,
then for $s > 0$, we have
$
\E{ \Exp{ s( f(\X) - \E { f( \X ) } )} }
= \E{ \Exp{ \sum_{ k \in [K] } s w_k f( \X_{F_k} ) } }.
$
Let $z_1, \ldots, z_K$ be any set of $K$ positive reals that sum to $1$.
By the Jensen's inequality,
\begin{align*}
\E{ \Exp{ s( f(\X) - \E { f( \X ) } )} }
&\le \E{ \sum_{ k \in [K] } z_k \Exp{ \frac{ s w_k }{ z_k } 
( f( \X_{F_k} ) - \E{ f( \X_{F_k} ) } ) } } \\
&= \sum_{ k \in [K] } z_k \E{ \Exp{ \frac{ s w_k }{ z_k } 
(f( \X_{F_k} ) - \E{f( \X_{F_k} )}) } },
\end{align*}
where the equality is due to the linearity of expectation.
Then Lemma \ref{martingale differences lemma} and Lemma \ref{McDiarmid's Lemma} give
\begin{align}
\E{ \Exp{ s( f(\X) - \E { f( \X ) } )} } 
\le \sum_{ k \in [K] } z_k  \Exp{ \frac{ s^2 w_k^2 }{ 8 z_k^2 } 
\left( \sum_{\{i, j\} \in E(G[F_k])} ( c_i + c_j )^2 
+ \sum_{T \in \T(F_k)} c_{\min,k,T}^2 \right) }.
\end{align}
Next, for all $k \in [K]$, we choose 
\begin{align*}
z_k = \frac{w_k}{Z} \sqrt{ 
\sum_{\{i, j\} \in E(G[F_k])} ( c_i + c_j )^2 
+ \sum_{T \in \T(F_k)} c_{\min,k,T}^2 },
\end{align*}
with
\begin{align*}
Z = \sum_{k \in [K]} w_k \sqrt{
\sum_{\{i, j\} \in E(G[F_k])} ( c_i + c_j )^2 
+ \sum_{T \in \T(F_k)} c_{\min,k,T}^2 }.
\end{align*}
Hence we have
$
\E{ \Exp{ s( f(\X) - \E { f( \X ) } )} }
\le \Exp{ s^2 Z^2 / 8 }.
$
Combining with Lemma~\ref{McDiarmid's Lemma} completes the proof.
\end{proof}

\section{Discussions}
\label{discuss}

We establish bounded difference inequalities for forest-dependent random variables;
it is unclear whether the proof based on coupling technique
can be adapted to the case of general graph-dependent case
without imposing decomposability constraints to the functions.
Some heuristic ideas of transforming the general graph to a forest
are given in \cite{zhang2019mcdiarmid}, 
however, various ad hoc constructions are needed for different graphs.
%We wish to obtain general results.
For the summation, which satisfies the decomposability constraint,
we obtain a better bound under graph-dependence than Janson's bound.
The direct application of Theorem \ref{forest} under forest-dependence
for summation may also give better results than Janson's, see Remark \ref{re}.

The (fractional) forest vertex covering used in Subsection \ref{frac} 
closely relates to (fractional) vertex arboricity.
Given a graph $G$,
the vertex arboricity $a(G)$ is the minimum
number of subsets into which the vertex set $V (G)$ can be partitioned
so that each subset induces an acyclic subgraph,
and the fractional vertex arboricity
$a_f(G)$ is the minimum of $\sum_{k} w_{k}$
over $ \textsf{FFC}(G) $.
Other upper bounds on $D(G, \c)$ in $(\ref{T})$ 
can also be obtained via fractional vertex arboricity
following Janson's approach using the Cauchy-Schwarz inequality,
see Remark \ref{jan}.

Other dependence characterizations for concentration
widely used in random fields and statistical physics 
are various dependency matrices, 
which quantify the strength of dependence among variables.
These matrices include Dobrushin interdependence matrix \cite{dobruschin1968description, chatterjee2005concentration}
and other mixing-based dependency matrices \cite{samson2000concentration,
kontorovich2008concentration, paulin2015concentration}.
% One classical result introducing $\eta$-mixing coefficients used in \cite{}
% establishes the following concentration inequality:
% \begin{align}
% \P { f(\X) - \E { f(\X) } \ge t }
% \le \exp \left( - \dfrac{2t^2}{ \| \Gamma \|^2_\infty \| \c \|^2_\infty } \right).
% \end{align}
% For extensive discussions and comparisons, see \cite{,}.
To employ their results, suitable estimates of the mixing coefficients
are required for specific applications,
which might not be handy in the combinatorial applications.
On the other hand, the coupling construction in this note 
can be used for the estimation of mixing coefficients, 
see \cite{kontorovich2017concentration}.

\section*{Acknowledgements}

The author would like to thank his supervisor Nick Wormald for valuable comments and discussions.
The author also thanks David Wood for a helpful discussion on graph coverings.
%The author is grateful to David Wood for mentioning the forest covers and the arboricity of a graph.
The author is grateful to an anonymous referee 
for comments, which lead to improvements of Subsection \ref{sec-m}.
%The author also thanks another anonymous referee for mentioning \cite{huang2021nonlinear}.
Part of this work was done when the author was 
at the Institute of Computing Technology, Chinese Academy of Sciences
under the supervision of Xingwu Liu,
to whom we are grateful for discussions.
Some results appeared in \cite{zhang2019mcdiarmid} without proofs included.
This research did not receive any specific grant
from funding agencies in the public, commercial, or not-for-profit sectors.

\bibliographystyle{plain}
\bibliography{ref}

\end{document}